\documentclass[a4paper,14pt]{article}
\usepackage{amsfonts,amsmath,amsthm,amscd,amssymb,latexsym}

\usepackage[T2A]{fontenc}
\usepackage[cp1251]{inputenc}
\usepackage[english ]{babel}  
\usepackage{graphicx}
\usepackage{setspace}
\usepackage{xcolor}
\usepackage{fancyhdr}
\usepackage{amsmath}
\usepackage[all, matrix,arrow,curve]{xy}
\usepackage{amsmath, amssymb, amsfonts, amsthm}
\usepackage{tikz-cd} 

\begin{document}

\title{On minimal involutive generic sets\\ of Extended Special Linear group $ESL_n(\mathbb{F}_p)$}

\author{R. Skuratovskii \\ {ruslan.skuratovskii@imath.kiev.ua} \\
Institute of applied mathematics and mechanics \\
of the NASU, Kiev, Ukraine, skuratovskii@nas.gov.ua}

\theoremstyle{plain}

\newcommand{\ssty}{\scriptstyle}
\newcommand{\w}{\omega}
\newcommand{\sr}{\stackrel}
\newcommand{\ov}{\overline}
\newcommand{\ga}{\gamma}
\newcommand{\al}{\alpha}
\newcommand{\be}{\beta}
\newcommand{\de}{\delta}
\newcommand{\si}{\sigma}
\newcommand{\la}{\lambda}

\newcommand{\N}{\mathbb{N}}
\newcommand{\Z}{\mathbb{Z}}
\newcommand{\R}{\mathbb{R}}
\renewcommand{\C}{\mathbb{C}}

\newcommand{\PP}{\mathcal{P}}

\newtheorem{theorem}{Theorem}[section]
\newtheorem{lemma}{Lemma}[section]
\newtheorem{proposition}{Proposition}[section]
\newtheorem{corollary}{Corollary}[section]
\newtheorem{definition}{Definition} [section]

\newtheorem{example}{Example}[section]
\newtheorem{remark}{Remark}[section]
\newtheorem{defn}{Definition}[section]
\newcommand{\keywords}{\textbf{Key words.  }\medskip}
\newcommand{\subjclass}{\textbf{MSC 2000. }\medskip}
\renewcommand{\abstract}{\textbf{Abstract.  }\medskip}
\numberwithin{equation}{section}

\setcounter{section}{0}
\renewcommand{\thesection}{\arabic{section}}
\newcounter{unDef}[section]
\def\theunDef{\thesection.\arabic{unDef}}

\maketitle

\def\Xint#1{\mathchoice
   {\XXint\displaystyle\textstyle{#1}}
   {\XXint\textstyle\scriptstyle{#1}}
   {\XXint\scriptstyle\scriptscriptstyle{#1}}
   {\XXint\scriptscriptstyle\scriptscriptstyle{#1}}
   \!\int}
\def\XXint#1#2#3{{\setbox0=\hbox{$#1{#2#3}{\int}$}
     \vcenter{\hbox{$#2#3$}}\kern-.5\wd0}}
\def\dashint{\Xint-}

\newcommand{\M}{{\cal M}}
\newcommand\disk{{\Bbb D}}

\def\esssup{\mathop{\rm {ess\,sup}}\limits}
\def\ink{\mathop{\int\int}\limits}
\def\h{{\Bbb H}}
\def\R{{\rm Re\,}}
\def\I{{\rm Im\,}}
\def\mod{{\rm mod\,}}
\def\e{\varepsilon}
\def\QED{{\hfill $\Box$\par\bigskip}}

\pagestyle{myheadings}
{\bf Abstract }
The size of minimal systems of generators and these systems themselves for groups $ESL_n \left[ \mathbb{Z} \right]$ and $ESL_n \left( \mathbb{F}_p \right)$ were found.
 Moreover we obtain triples of involutions with\textit{\textbf{ two commuting involu\-tions}} generating $ESL_5 \left[ \mathbb{Z} \right]$, $ESL_5 \left[ \mathbb{Z} \right]$ relatively,  that are Mazurov triples.
    \begin{keywords}
     Extended special linear group, minimal involutive generating set with two commuting involutions, $sggi$ groups, minimal generating set, \\
\end{keywords}

~\\

{\bf Introduction }
In this research we continue our previous investigation \cite{SkuESL}, where we
generalized the group of unimodular matrices \cite{Amit} and found its structure just for the case $n=2$. For this aim we construct split extension of $SL_n\left( \mathbb{F}_p \right)$ by arbitrary matrix $D \in M_n(\mathbb{F}_p)$ having $\det \left( {{D}} \right)=-1$ to $ESL_n \left[ \mathbb{Z} \right]$.
Similar arguments are fair for constructing  $ESL_n \left[ \mathbb{Z} \right]$ over the integral ring.

For this goal we propose one extension of the special linear group.
Groups generated by three involutions, two of which are permutable, have long been of interest in the theory of matrix groups \cite{Maz}, for instance such generating set was researched for $S{{L}_{2}}[{{\mathbb{Z}+ i\mathbb{Z}}}]$ \cite{Nuz}.

The matrices involutions of $ES{{L}_{2}}\left[ \mathbb{Z} \right]$ can be a presentation of reflections from the Coxeter group $(W,V)$ \cite{Coxet} which map domains in relation to the hyperplanes $F_j$ corresponding to integer value in $e^{2ipx}$ which maps $\mathbb{R}$ in $S^1$.
Thus, involutions which frames a generator system of $ES{{L}_{2}}\left[ \mathbb{Z} \right]$ are related with generator system of $W$ \cite{Coxet} due to the transformation to a correspondent basis.
These reflections commute if indexes of their preimages $\sigma_i, \sigma_j$ from Artin braid $n$-string group satisfy the inequality $|i - j|>2$. This induce problem of matrix presentation of such reflections, which was stated in the paper \cite{Leem} as separated problem for matrix groups which were called $sggi$ groups.


The question of involutions $g_i$ commutation that appear as reflections with respect to hyperfaces in rectangle Coxeter group $RC_p$, as such the generators are involutions, is described in article \cite{TP}, where it is stated that commutation exists only when $ F_i \cap F_j \neq   \emptyset $. Therefore it is important to research matrix representation of generator set of involutions in the provided order with commutation of neighbouring involutions in linear group over $\mathbb{Z}$.
Therefore, it is important to study the matrix representation of the set of involutive generators of the graph $G$ with commutation of adjacent involutions in some order.

The question of study an automorphism group of lattice $\mathbb{Z}^n$ is very important because it is an algebraic foundations of lattice-based cryptography \cite{LatCry, NTRU, BookLatCry}.
Group of automorphisms of $\mathbb{Z}^n$ is much bigger than just signed permutations, it's infinite (for $n \geq 2$) and consists of all integer matrices whose determinant is 1 or -1.
An automorphism of this lattice be an orthogonal transformation (i.e. linear transformation that preserves the inner product).

The automorphism group of $\mathbb{Z}^n$ is not finite when $n \geq 2$, since it consists of all $n\times n$ matrices with integer coefficients and determinant 1 or -1 that means this automorphism group is $ES{{L}_{n}}[{{\mathbb{Z}}}]$.

\section{Concept of $ESL_3(\mathbb{F}_p)$ and $ESL_3(\mathbb{Z})$}

We recall a concept of $ES{{L}_{2}}\left( \mathbb{F}_p \right)$ introduced in our previous work \cite{SkuESL}.
\begin{definition}
 The set of matrices
\begin{equation}\label{ESL}
 \left\{ {{M}_{i}}:\det({{M}_{i}})=\pm 1,  {M}_{i} \in GL_2(\mathbb{F}_p) \right\}
 \end{equation}
forms the extended special linear group and is denoted by ${ESL}_2(\mathbb{F}_p)$.
\end{definition}
Let $SL_3(\mathbb{Z})$ or $SL_3(\mathbb{F}_p)$ denotes the special linear group of degree 3 over integer ring of a finite field, respectively.
\begin{definition}
\textit{ The set of matrices over the integer ring or over a field \[\left\{ {{M}_{i}}:\det({{M}_{i}})=\pm 1,  {M}_{i} \in GL_3(\mathbb{Z}), \, or \ {M}_{i} \in GL_3(\mathbb{F}_p) \right\} \] forms \textbf{extended special linear group} over $\mathbb{Z}$ and is denoted by ${ESL}_3(\mathbb{Z})$ or by ${ESL}_3(\mathbb{F}_p)$, respectively.}
\end{definition}
We note that $ESL_3(\mathbb{Z}) \simeq GL_3(\mathbb{Z})$ however, this isomorphism turns into an isomorphic embedding $ESL_3(\mathbb{F}_p) < GL_3(\mathbb{F}_p)$ over a finite field $\mathbb{F}_p$.


To construct split extension of $SL\left( 3, \mathbb{F}_p \right)$ we consider the following class of matrix having $\det \left( {{D_{1}}} \right)=-1$,
but do not centralizing the group $S{{L}_{3}}\left( \mathbb{F}_p \right)$.
$${{D}_{1}}=\left(\hspace{-1.5mm} \begin{array}{rrr}
   -1 & 0 & 0  \\
   0 & 1 & 0  \\
   0 & 0 & 1  \\
\end{array} \right), \
 {{D}_{2}}=\left( \begin{matrix}
   1 & 0 & 0  \\
   0 & -1 & 0 \\
   0 & 0 & 1  \\
\end{matrix} \right),
         {{D}_{3}}=\left( \begin{matrix}
   1 & 0 & 0  \\
   0 & 1 & 0  \\
   0 & 0 & -1  \\
\end{matrix} \right).$$

Based on the above, we conclude that structure of group generated as extension of $S{{L}_{3}}\left( \mathbb{F}_p \right)$ by $\left\langle {{D}_{1}} \right\rangle $ is a semidirect product
$$\left\langle {{D_{1}}} \right\rangle \ltimes SL_{3}\left( \mathbb{F}_p \right)\simeq ES{{L}_{3}}\left( \mathbb{F}_p \right)$$
with a kernel $S{{L}_{3}}\left( \mathbb{F}_p \right)$.
This group is endowed with the same structure over the integer ring $\left\langle {{D_{1}}} \right\rangle \ltimes SL_{3}\left( \mathbb{Z} \right)\simeq ES{{L}_{3}}\left( \mathbb{Z} \right)$.

But if we extend the kernel by $D_{1,2,3}=D_1D_2 D_3$, which centralizes $SL (3,\mathbb{F}_p)$ as a scalar matrix, then the semidirect product degenerates into a direct product
$$\langle D_{1,2,3} \rangle \times SL_3 (\mathbb{F}_p) \simeq ESL_3 (\mathbb{F}_p).$$

The role of a complementary subgroup for the kernel $SL(3,F)$ in the split group $ES{{L}_{3}}(\mathbb{F}_p)$ can also be played by subgroups formed by $D_2$ and $D_3$ therefore $\left\langle {{D_{2}}} \right\rangle \ltimes SL_{3}\left( \mathbb{F}_p \right)\simeq ES{{L}_{3}}\left( \mathbb{F}_p \right)$ and $\left\langle {{D_{3}}} \right\rangle \ltimes SL_{3}\left( \mathbb{F}_p \right)\simeq ES{{L}_{3}}\left( \mathbb{F}_p \right)$.
A diagonal matrix with one diagonal element $ -1$ and the rest of them 1 be called \textit{\textbf{elementary diagonal}} matrix.


\begin{remark}
{\sl If $n=2k+1$, then there exist exactly $2^{n-1}$ diagonal extensions by subgroup generated by elementary diagonal matrices of the group $SL(n, \mathbb{F}_p)$ to the group $ESL(n, \mathbb{F}_p)$. }
 \end{remark}
\begin{proof}
To generate diagonal matrix with odd number of -1 by using elementary diagonal matrices $D_i$ we can multiplicate odd number 5 of them. The quantity of diagonal matrix with odd number of -1 on diagonal is sum of combinations of $n$ by odd numbers means coordinates of -1 on diagonal.
In order to compute a sum of ${n{} \choose 2l}$ consider a sum of $(1+1)^n=2^n$ and $(1-1)^n=0$ which is $ \left({n^{ } \choose 0}+{n^{ } \choose 2}+ {n^{} \choose 4}+ \ldots \right)$ and divide it by 2.

\begin{equation}\label{placement}
\left(\sum_{l=0}^{{n^{}}}{n^{} \choose 2l}(1+(-1)^l)\right):2 =
2^{n^{}-1}.
\end{equation}
\end{proof}

Let $SL_3(\mathbb{Z})$ denote the special linear group of degree 3 over integer ring.
By transvection $t_{ij}$ we mean the sum $E+e_{ij}$, where $e_{ij}$ is a matrix unit with 1 only in  intersection of $i$-th row and $j$-th column the rest elements are 0.

Denote a permutation matrix of order 3 by ${{P}_{3}}$ and the transvection \cite{SkuESL} by $t_{12}$ of group $SL_3(\mathbb{F}_p)$.

By the transvection $t_{ij}$ we mean the sum $E+e_{ij}$, where $e_{ij}$ is a matrix unit with 1 only in  intersection of $i$-th row and $j$-th column the rest elements are 0.   


\textbf{Proposition 1.}
{\sl  Minimal generating set for $ESL_3(\mathbb{Z})$ consists of 2 generators:
\begin{center}
  $P=\left( \begin{matrix}
   0 & -1 & 0  \\
   0 & 0 & 1  \\
   1 & 0 & 0  \\
\end{matrix} \right)$ and  $t_{12}=\left( \begin{matrix}
   1 & 1 & 0  \\
   0 & 1 & 0  \\
   0 & 0 & 1  \\
\end{matrix} \right).$
\end{center} }
The order of the permutation matrix is indicated by the relation $P^6=E$.
The size of the generic set is minimal for non-cyclic group so its minimality does not order a proof. 

For convenience we fix some notations for \textit{diagonal involutive matrices} from $ESL_5(\mathbb{Z})$:
\begin{center}
$I_{12} = \begin{pmatrix}
            -1 & 1 & 0 & 0 & 0 \\
            0  & 1 & 0 & 0 & 0 \\
            0  & 0 & 1 & 0 & 0 \\
            0  & 0 & 0 & 1 & 0 \\
            0  & 0 & 0 & 0 & 1
        \end{pmatrix},
        I_{23} = \begin{pmatrix}
            1 & 0 & 0 & 0 & 0 \\
            0 & -1& 1 & 0 & 0 \\
            0 & 0 & 1 & 0 & 0 \\
            0 & 0 & 0 & 1 & 0 \\
            0 & 0 & 0 & 0 & 1
        \end{pmatrix}, \ldots,
        I_{45} = \begin{pmatrix}
            1 & 0 & 0 & 0 & 0 \\
            0 & 1 & 0 & 0 & 0 \\
            0 & 0 & 1 & 0 & 0 \\
            0 & 0 & 0 & -1& 1 \\
            0 & 0 & 0 & 0 & 1
        \end{pmatrix}.
    $
    \end{center}

\begin{example}\label{inv} Note that it is possible to express transvection using only two non-commutative involutions.  
\begin{center}
${{i}_{12}}=\left( \begin{matrix}
   1 & 1 &  0  \\
   0 & -1 & 0  \\
   0 & 0 & -1  \\
\end{matrix} \right)$,  ${{i}_{23}}=\left( \begin{matrix}
   1 & 0 & 0  \\
   0 & 1 & 1  \\
   0 & 0 & -1  \\
\end{matrix} \right)$,
$({{i}_{12}}{{i}_{23}})^2=\left( \begin{matrix}
   1 & 1 & 1  \\
   0 & -1 & -1  \\
   0 & 0 & 1  \\
\end{matrix} \right)^2= \left( \begin{matrix}
   1 & 0 & 1  \\
   0 & 1 & 0  \\
   0 & 0 & 1  \\
\end{matrix} \right) ={{t}_{31}}^{}$.
 \end{center}
Thus, we generate transvection $t_{31}$ by two involutions $i_{12}$ and $i_{23}$. The rest of two transvections can be constructed as a product of involutions by similar arguing with $i_{kl}$. The same is true for $ESL_n(\mathbb{Z})$, $n \geq 3$.
\end{example}

The existence of a non-trivial homomorphism $\varphi :\,\,{{\mathbb{Z}}_{2}}\to Aut\left( S{{L}_{2}}(\mathbb{Z}) \right)$, as well as $\phi :\,\,{{\mathbb{Z}}_{2}}\to Aut\left( S{{L}_{2}}({{\mathbb{F}}_{p}}) \right)$ can be proved by indicating an element of order 2 in the automorphisms of base group that is the kernel of the semidirect product we want to construct.

There is countergradient automorphism in $S{{L}_{3}}\left( \mathbb{Z} \right)$, namely, $\varphi :\,M\to {{\left( {{M}^{T}} \right)}^{-1}}$ or an alternating automorphism of order 2 acting by conjugating $\varphi :\,M\to D_1^{-1}MD_1$, which is called the diagonal automorphism \cite{Mersl}.

Recall the \textbf{definition} of $\mathbf{TI-subgroup}$ \cite{Suds,Zu}.  Let $G$ be a group and $A < G$, then $A$ is called $\mathbf{TI-}$subgroup iff  $A \cap A^g = e$ for each $g \in G \setminus N_G(A)$.

\begin{remark}
Subgroup $\mathbb{C}_{2}$ is $\mathbf{TI-subgroup}$ and antinormal subgroup.
\end{remark}

\begin{proof}
In view of ${\mathbb{C}}_{2}$ is one generated then its centralizer coincides with its normalizer. One easy can verify that centralizer consists of all diagonal matrices from $ESL_2(\mathbb{F}_p)$.
Let us find a structure of such normalizer $N_{ESL_2(\mathbb{F}_p)} ({\mathbb{C}}_{2})$.
In view of e.v. is  invariant under conjugation by non-singular matrix over field the normalizer of top subgroup ${\mathbb{C}}_{2}$ in $ESL_2(\mathbb{F}_p)$ consists of  all diagonal matrices from $ESL_2(\mathbb{F}_p)$ and permuta\-tio\-nal matrix ${{\mathcal{P}}}=\left( \begin{array}{rr}
   0\,\,\, &1 \\
  1\,\, & 0 \\
\end{array} \right)$.
We assume that $N_{ESL_2(\mathbb{F}_p)} ({\mathbb{C}}_{2}) \simeq   D(SL_2(\mathbb{F}_p))\rtimes \,{\mathcal{P}}$, where $D(SL_2(\mathbb{F}_p))$ diagonal subgroup of $ESL_2(\mathbb{F}_p)$.

For the rest of elements condition of $A \cap A^g = e$ for each $g \in ESL_2(\mathbb{F}_p) \setminus N_{ESL_2(\mathbb{F}_p)} ({\mathcal{C}}_{2})$ holds. Thus, $\mathbb{C}_{2}$ is $\mathbf{TI-subgroup}$, hence $\mathbb{C}_{2}$ is antinormal subgroup.
\end{proof}

Let us find a normal closure of $D_i$ for $1 \leq i \leq 3$ in $ESL(3, \mathbb{Z})$ which be denoted by ${{N}_{ESL}}\left( {{D}_{1}} \right)$. We demonstrate two typical classes of ${{N}_{ESL_3(\mathbb{Z})}}\left( {{D}_{i}} \right)$ normal closure here:
\[{{N}_{ESL_3(\mathbb{Z})}}\left( {{D}_{1}} \right)=\left\{ \left( \begin{matrix}
   d & 0 & 0  \\
   0 & a & b  \\
   0 & g & c  \\
\end{matrix} \right) \in ESL_3(\mathbb{Z}) \left| d,a,b,c,g\in \mathbb{Z} \right. \right\}\]
\[{{N}_{ESL_2}}\left( {{D}_{2}} \right)=\left\{ \left( \begin{matrix}
   x & 0 & f  \\
   0 & y & 0  \\
   c & 0 & z  \\
\end{matrix} \right)\in ESL_3(\mathbb{Z}) \left| x,y,f,c,z\in \mathbb{Z} \right. \right\}. \]







The intersection $A \cap A^g = E$, provided $g \notin {{N}_{ESL_3(\mathbb{Z})}}\left( {{D}_{i}} \right)$ for each $1 \leq i \leq 3$, is trivial by virtue of ${{N}_{ESL_3(\mathbb{Z})}}\left( {{D}_{i}} \right)={{C}_{ESL}}_3({D}_{i})$ (a normalizer of a one-generated group is equal to its centralizer).

The normalizer ${{N}_{ESL_3(\mathbb{Z})}}\left( {{D}_{123}} \right)$ of the subgroup $\left\langle D_{123} \right\rangle $ is an exception as an element of the center of the entire group $ESL_3(\mathbb{Z})$, therefore ${{N}_{ESL_3(\mathbb{Z})}}\left( {{D}_{i}} \right)$, therefore ${{N}_{ESL_3(\mathbb{Z})}}\left( {{D}_{123}} \right) = ESL_3(\mathbb{Z})$ and consiquently $ \left\langle {D_{123}} \right\rangle^g = \left\langle D_{123} \right\rangle$. Hence, $\left\langle D_{123} \right\rangle \cap \left\langle D_{123} \right\rangle ^g = \left\langle D_{123}\right\rangle$.

Let $i$ denotes embedding of $SL(3, \mathbb{Z})$ in $ESL(3,\mathbb{Z})$.
where epimorphism $\eta$ is homomorphism with the kernel $N = \left\langle E, -E \right\rangle$. Thus we have quotient ${}^{PESL(3,\mathbb{Z})} \diagup{{}_{ \left\langle E, -E \right\rangle }}\simeq PSL(3,\mathbb{Z})$ similarly as       ${}^{ESL(3, \mathbb{Z})}/{}_{ \left\langle E, -E \right\rangle } \simeq PESL(3,\mathbb{Z})$.
In the general case $n=2k+1$, there is an isomorphism $SL(2k+1, \mathbb{Z}) \simeq PSL(2k+1, \mathbb{Z})$ due to the absence of scalar matrices $A_i \in SL(2k+1, \mathbb{Z})$ with $\det(A_i)=-1$.

\begin{remark}\label{SLn=2k+1}
In the general case $n=2k+1$, there is an isomorphism $SL(2k+1,Z) \simeq PSL(2k+1,Z)$ due to the absence of non-trivial scalar matrices $A_i$ with $\det(A_i)=-1$
\end{remark}
\begin{proof}
    The kernel of homomorphism from $SL(3,\mathbb{Z})$ to $PSL(3,\mathbb{Z})$ consists only of scalar matrix $E$ due to the absence of
another scalar matrices $A_i$ with $\det(A_i)=1$ over $\mathbb{Z}$.  Therefore $SL(3,\mathbb{Z}) \simeq PSL(3, \mathbb{Z})$.
\end{proof}


In even dimensions, we have a commutative diagram of morphisms
where epimorphism ${i}$ provide us the quotient ${}^{ESL(2,\mathbb{Z})}/{}_{ \left\langle E, -E \right\rangle } \simeq PESL(2, \mathbb{Z})$.

Let $\phi$ denotes embedding of $SL(2,\mathbb{Z})$ in $ESL(2, \mathbb{Z})$.
\begin{center}
$\xymatrix{\ar @{} [dr] |{}
SL(2,\mathbb{Z}) \ar @{^{(}->}^{\phi} [r] \ar@{.>}[dr]|-{(i)} \ar @{->>} [d] ^{\alpha} & ESL(2,\mathbb{Z}) \ar @{->>}  [d] ^{\psi} \\
PSL(2,\mathbb{Z})  \ar @{^{(}->}^{\xi} [r]  & PESL(2,\mathbb{Z})  }$
\end{center}



\begin{remark}
 There isn't is not surjective homomorphism from $ESL_n(\mathbb{Z})$ to $SL_n(\mathbb{Z})$ for $n=2$, but for $n \geq 3 $ such surjective homomorphism exists.
\end{remark}
\begin{proof}
Taking into account that $ESL_2(\mathbb{Z})$ can be generated by involutions as in the Example \eqref{inv}, then all its elements under homomorphism $\phi'$ from $ESL_2(\mathbb{Z})$ to $SL_2(\mathbb{Z})$ map in elements of second order in $SL_2(\mathbb{Z})$, but there are only $E$ and $-E$ of order in $SL_2(\mathbb{Z})$. Thus, there is not surjective homomorphism from $ESL_2(\mathbb{Z})$ to $SL_2(\mathbb{Z})$.  The subgroup generated by involutions in $SL_2(\mathbb{Z})$ is subgroup $\langle E, -E \rangle$, for $n=2$ which does not generate $SL_2(\mathbb{Z})$.

However, $ESL(n, \mathbb{Z})$ contains sufficient enough number of involutions to generate it, which be proved in Proposition \ref{I_{12},D_{1}, F_L}. Hence aforementioned homomorphism already exists.
\end{proof}

\begin{definition} \label{quasimple}
 A group $G$ is called quasimple if its inner automorphism group $Inn(G)$ is simple.
\end{definition}

\begin{definition}
      We define $PESL(n,p)$ as a quotient of $ESL(n,p)$ by its center.
\end{definition}


%




\begin{proposition}
   {\sl  If $-1$ is not a quadratic residue in $F_p$, provided $p=4k+3$ and $p>2$, then $Z(ESL(2,p)) \simeq Z(SL(2,p))$ for $k \in \mathbb{N}$, and furthermore, $PSL\left( 2, p \right)\triangleleft PESL\left( 2, p \right)$.

In~ the~ case~ $p=2$,~  $PESL\left( 2k, 2 \right) = PSL\left( 2k, 2 \right)$~ and~ an~ index \-
$\left[ Z\left( ESL\left( 2k, 2 \right) \right):Z\left( SL\left( 2k, 2 \right) \right) \right]=2$.}
\end{proposition}

\begin{proof}
Since -1 is a square residue in $F_p$, then the equation ${{a}^{2}}=-1$, where $k>1$ has only trivial solutions in ${{\text{F}}_{p}}$,  the center of $ESL\left( n,p \right)$ is the same as $\left[ Z\left( ESL\left(n, p \right) \right)\,\,:\,\,Z\left( SL\left( n,p \right) \right) \right]=2$. As a consequence, $Z(ESL(2,p)) \simeq Z(SL(2,p))$ and  for the quotient groups $[PESL\left( n,p \right): PSL\left( n,p \right)]=2$ holds.
\end{proof}

\begin{theorem}
 For $n=2$, provided $\left( \frac{-1}{p} \right)=1$, which corresponds to $p=4k+1$, we have $\left[ Z\left( ESL\left( n,p \right) \right)\,\,:\,\,Z\left( SL\left( n,p \right) \right) \right]=2$, and $PESL\left( n,p \right)=PSL\left( n,p \right)$. Furthermore, $PESL\left( n, p\right)$ is simple, except special cases of $(n, p)$.
\end{theorem}

\begin{proof}
In view of the equation ${{a}^{2}}=-1$, where $k>1$ has non-trivial solutions in ${{\text{F}}_{p}}$, since -1 is a square residue, the center of $ESL\left( n,p \right)$ has become twice as large as $\left[ Z\left( ESL\left(n, p \right) \right)\,\,:\,\,Z\left( SL\left( n,p \right) \right) \right]=2$. As a consequence, the quotient groups $PESL\left( n,p \right)=PSL\left( n,p \right)$.
\end{proof}






\begin{center}
Diagram, for the case $ESL\left( n, q\right)$, $\mathbb{F}_q, q=p^m$,  $(-1)^{\left(\frac{q-1}{\gcd(n, q-1)} \right)}=1$, $n= 2k$
\end{center}
\begin{equation}\label{diag:my_diagram}
\xymatrix{
SL(n, \mathbb{F}_q) \ar @{^{(}->}@<-1ex>^{\phi} [r] \ar@{.>}[dr]|-{(i)}  \ar @{->>} [d] ^{\rho }
& ESL(n,\mathbb{F}_q)  \ar @{->>}  [d] ^{\psi}  \\
PSL(n,\mathbb{F}_q)                       & PESL(n, \mathbb{F}_q) \ar @{<<->>}[l]_{=} }
\end{equation}
 An epimorphism $\rho$ has kernel subgroup of scalar matrices with $det(A)=1$.

\begin{theorem}
 For $n=2k$, $q=p^m$, provided $(-1)^{\left( \frac{q-1}{gcd(q-1,n)} \right)}=1$, we have $\left[ Z\left( ESL\left( n,q \right) \right)\,\,:\,\,Z\left( SL\left( n,q \right) \right) \right]=2$, and $PESL\left( n,q \right)=PSL\left( n,q \right)$.
\end{theorem}

\begin{proof}
The center $Z\left( ESL\left( n,q \right) \right)$ consists of scalar matrices so we consider an equation $x^n=\pm 1$.

The proof is based of the Fermat's theorem and the fact that multiplicative group of $\mathbb{F}_{q}$ is cyclic so $g^{{p-1}}=1$, therefore  $g^{\frac{p-1}{2}}=-1$ where $g$ is generator. This entails $(g^{\left( \frac{q-1}{2gcd(q-1,n)} \right)})^n=-1$ which implies that $\frac{q-1}{gcd(q-1,n)}$ is even too $\frac{q-1}{2gcd(q-1,n)} \in \mathbb{N}$.

The necessary of this condition follows from that solution of equation $x^m=-1$ and from the equation obtained by exponenting of this equation to $\frac{q-1}{2gcd(q-1,m)}$ power $(x^m)^{\left(n \frac{q-1}{2gcd(q-1,m)} \right)}=(x^m)^{t(p-1)}= 1$ by Fermat's theorem. Therefore fraction $\frac{q-1}{gcd(q-1,m)}$ have to be even for existence of solution in $\mathbb{F}_q$.
\end{proof}

\begin{proposition}
    If $-1$ is not a quadratic residue in $F_p$, provided $p=4k+3$ and $p>2$, then  $Z(ESL(2,p)) \simeq Z(SL(2,p))$ for $k \in \mathbb{N}$, and furthermore, $PSL\left( 2, p \right)\triangleleft PESL\left( 2, p \right)$.

In~ the~ case~ $p=2$,~  $PESL\left( 2k, 2 \right) = PSL\left( 2k, 2 \right)$~ and~ an~ index \-
$\left[ Z\left( ESL\left( 2k, 2 \right) \right):Z\left( SL\left( 2k, 2 \right) \right) \right]=2$.
\end{proposition}
 \begin{proof}
Since the center of $ESL(2,p)$ consists of scalar matrices, it can contain only matrices of the form
$\left( \begin{matrix}
   \alpha & 0  \\
   0 & \alpha  \\
\end{matrix} \right)$, but in this case $-1$ is not a residue in ${{\text{F}}_{p}}$, there are no solutions in $F_p$ for the equation ${{\alpha}^{2}}\equiv -1(\bmod p)$. Consequently, there are only scalar matrices $\alpha E$ having $\det (A) =\alpha^2 =  1$ in the center of $Z(ESL_2(\mathbb{F}_p))$, that is why it coincides with $Z(SL_2(\mathbb{F}_p))$,
this yields $PSL\left( 2k,p \right)\triangleleft PESL\left( 2k,p \right)$ as a subgroup of indeed 2. The same is true in any even dimension $n=2k$.

In the case $p=2$, taking into account that in $\mathbb{F}_2$ each element is a square residue, the equation $\alpha^2 = - 1$ mentioned above is solvable, that entails a doubling of the center $Z(ESL_2(\mathbb{F}_p))$. Now it is obvious that \\ $\left[ Z\left( ESL\left( 2k, 2 \right) \right): Z\left( SL\left( 2k, 2 \right) \right) \right]  =2$ and as a consequence $PESL\left( 2k, 2 \right)$ as a quotient of $ESL\left( 2k, 2 \right)$ by its center decreases twice, therefore $PESL\left( 2k, 2 \right) = PSL\left( 2k, 2 \right)$.
\end{proof}

\begin{center}
\textbf{For $\left( \frac{-1}{p} \right)=-1$ (The case $n= 2k$)}
\end{center}
\begin{equation}\label{diag:my_diagram}
\xymatrix{
SL(n, \mathbb{F}_p) \ar @{^{(}->}@<-1ex>^{\phi} [r] \ar@{.>}[dr]|-{(i)}  \ar @{->>} [d] ^{\rho }
& ESL(n,\mathbb{F}_p)  \ar @{->>}  [d] ^{\psi}  \\
PSL(n,\mathbb{F}_p)   \ar @{^{(}->}@<-1ex>^{\xi} [r]                 & PESL(n, \mathbb{F}_p)  }
\end{equation}

\begin{theorem}\label{n=2k+1}
For $n=2k+1$ we have $PESL\left( n,p \right)=PSL\left( n,p \right)$ an index of a center $\left[ Z\left( ESL\left( 2k+1,p \right) \right):Z\left( SL\left( 2k+1,p \right) \right) \right]=2$. Furthermore, $PESL\left( n, p\right)$ is simple, except for special cases of $( n, p )$.
\end{theorem}
Taking into account that equation ${{a}^{2k+1}}=-1$ has non-trivial solutions in every ${{\text{F}}_{p}}$, $p>1$, then the number of scalar matrices with $\det(A)=-1$ is equal to the number of scalar matrices with $\det(A)=1$.
Then the number of scalar matrices with $\det(A)=-1$ coincides with the number of scalar matrices with $\det(A)=1$, therefore
$\left[  ESL\left( 2k+1,p \right): \\  SL\left( 2k+1,p \right) \right]=2$.  This determines the center index $\left[ Z\left( ESL\left( 2k+1,p \right) \right): \\ Z\left( SL\left( 2k+1,p \right) \right) \right]=2$.
Therefore, $PESL\left( n,p \right)=PSL\left( n,p \right)$ in this case.
Thus, if in the case $PSL\left( n,p \right)$ is simple, (here are exceptions $ PSL(2,2), PSL(2,3)$)  
then in accordance with Definition \eqref{quasimple} $ESL(2k+1,p)$ because $Inn (ESL(2k+1,p)) \simeq ESL(2k+1,p) / Z(ESL(2k+1,p)) \simeq PSL(2k+1,p)$.
Based on the evidence presented above, the following commutative diagrams are in place for $n=3$.
\begin{center}
\textbf{Commuting diagram for $n=2k+1$ is similar for the case $n= 3$}
\end{center}
\begin{equation}\label{diag:3,Z}
\xymatrix{
SL(3, \mathbb{Z}) \ar @{^{(}->}@<-1ex>^{\phi} [r] \ar@{.>}[dr]|-{(i)}   \ar @{<->}[d]_{\simeq}^{\rho}
& ESL(3,\mathbb{Z}) \ar @{->>}@<-1ex>[l]_{\mathfrak{C}} \ar @{->>}  [d] ^{\psi}  \\
PSL(3,\mathbb{Z})                        & PESL(3, \mathbb{Z}) \ar @{<<->>}[l]_{=} }
\end{equation}  
Here the equality $PESL(2k+1,\mathbb{Z}) = PSL(2k+1,\mathbb{Z})$ is provided via \\ $\left[ Z\left( ESL\left( 2k+1,p \right) \right):Z\left( SL\left( 2k+1,p \right) \right) \right]=2$ stated in the Theorem \ref{n=2k+1} and the same index of whole group $SL\left( 2k+1,p \right)$ namely \\ $\left[  ESL\left( 2k+1,p \right):  SL\left( 2k+1,p \right)  \right]=2$. The isomorphism $\rho$ is justified in Remark \ref{SLn=2k+1}.





\begin{center}
\textbf{For $\left( \frac{-1}{p} \right)=1$ (The case $n= 2$)}
\end{center}
\begin{equation}\label{diag:my_diagram}
\xymatrix{
SL(n, \mathbb{F}_p) \ar @{^{(}->}@<-1ex>^{\phi} [r] \ar@{.>}[dr]|-{(i)}  \ar @{->>} [d] ^{\rho }
& ESL(n,\mathbb{F}_p)  \ar @{->>}  [d] ^{\psi}  \\
PSL(n,\mathbb{F}_p)                       & PESL(n, \mathbb{F}_p) \ar @{<<->>}[l]_{=} }
\end{equation}


The homomorphism $\mathfrak{C}$ on the diagram \ref{diag:3,Z} from $ESL_{2k+1}(p)$ onto $SL_{2k+1}(p)$ exists due to the large variety of involutions in $SL_{2k+1}(p)$ that generate entire $SL_{2k+1}(p)$.
This surjection is equality $\mathfrak{C}(A) = A$ if $det(A)=1$ and $\mathfrak{C}(A) = -A$ if $det(A)=-1$.
 The involutive generating set of $ESL_{2k+1}(p)$ maps to the involutions of $SL_{2k+1}(p)$, and only one non-trivial involution $D_{1,2,3}$ is included in the kernel of this homomorphism $\psi$.

In exceptional case $n=2$ homomorphism $ESL_{2}(p) \xrightarrow{\chi} SL_{2}(p)$ maps this group in subgroup $\langle E, -E\rangle$ because there are no another elements of order 2 in $SL_{2}(\mathbb{Z})$ but only these two elements of order two are images of involutive generating set of $ESL_{2}(\mathbb{Z})$. Therefore all generators from involutive generating set of $ESL_{2}(\mathbb{Z})$ maps in the subgroup of two matrices $\langle E, -E\rangle$.

\section{Involutive generating set.}

\begin{proposition}
{\
The \textbf{minimal generating set} of $ESL_3[\mathbb{Z}]$ is $\langle {M}_{6}, t_{12}\rangle$
\begin{center}
  ${{M}_{6}}=\left( \begin{matrix}
   0 & 1 & 0  \\
   0 & 0 & 1  \\
   -1 & 0 & 0  \\
\end{matrix} \right)$ and $t_{12}=\left( \begin{matrix}
   1 & 1 & 0  \\
   0 & 1 & 0  \\
   0 & 0 & 1  \\
\end{matrix} \right),$
\end{center}
which possess the relations ${M}_{6}t_{12} {M}^{-1}_{6}=t^{-1}_{31}, \, {M}_{6}t_{31} {M}^{-1}_{6}=t^{-1}_{23} $, ${M}_{6}t_{23} {M}^{-1}_{6}=t^{-1}_{12}$, ${M}_{6}t_{13}{M}^{-1}_{6}=t^{-1}_{32}$, ${M}_{6}t^{-1}_{23} {M}^{-1}_{6}=t^{-1}_{31}$, ${M}^6_{6}=E$, ${M}^3_{6}=-E$, as well as the relations between transvections $[t_{ij}, t_{jk}]=t_{ik}$, wherein \, $i\neq k$, and \, $[t_{ij}, t_{kl}]=e$ for $i\neq l$ and $k \neq j$ }.
  \end{proposition}
\begin{example}
Let $G= \langle M_5, t_{13} \rangle$. The order of the generated group is 11232 that is 2 times greater than the order of $SL_3(\mathbb{F}_3)$, and coincides with an order of $ESL_3(\mathbb{F}_3)$. Moreover $G \simeq ESL_3(\mathbb{F}_3) \simeq GL_3(\mathbb{F}_3)$.
\end{example}

\begin{proposition}
The following special cases are true $ESL_n(\mathbb{F}_5) \lhd GL_n(\mathbb{F}_5)$,    $[GL_n(\mathbb{F}_4) : ESL_n(\mathbb{F}_4)]=3$.
\end{proposition}
\begin{proof}
    For $M \in GL_n(\mathbb{F}_5)$ the variety of determinant values is $V_5(GL_n)= \{1,-1,2,-2\}$ and for $M \in ESL_n(\mathbb{F}_5)$ this variety is $V_5(ESL)=\{1,-1\}$. The relation $V_5(GL_n) : V_5(ESL_n) =2$ that caused
$[GL_n(\mathbb{F}_5) : ESL_n(\mathbb{F}_5)]=2$. This completes the proof of $ESL_n(\mathbb{F}_5) \lhd GL_n(\mathbb{F}_5)$.

Similar reasoning $ESL_n(\mathbb{F}_4)$ leads us to the conclusion $[GL_n(\mathbb{F}_4) : ESL_n(\mathbb{F}_4)]=3$.
\end{proof}

Let $G= \langle M_5, t_{13} \rangle$. The order of the generated group is 11232 that is 2 times greater than the order of $SL_3(\mathbb{F}_3)$, and coincides with an order of $ESL_3(\mathbb{F}_3)$. Moreover $G \simeq SL_3(\mathbb{F}_3) \simeq GL_3(\mathbb{F}_3)$.

\begin{proof}
Proving the relations is a simple multiplication check. Due to the relations mentioned above all elementary transvections are presented in explicit form, for instance by the relation ${M}_{6}t_{12} {M}^{-1}_{6}=t^{-1}_{31}$ and the group axiom about an inverse element we get $t_{31}$. Thence
The transformation for well known generators \cite{Humph} allow us to generate $SL_3(\mathbb{F}_3)$.
 In view of Lemma \ref{negativedet} and generator $ M_5$ the generating set of $ESL_3[\mathbb{Z}]$ is constructed.
 The size of this generic set is minimal for non-cyclic group, so its minimality does not demand any proof.
\end{proof}



\begin{proposition}
    The involutive generating set for $ESL_5 (\mathbb{F}_2)$ consists of 4 matrices $i_{11}$, $i_{12}$, $D_0$, $B$. This set admits generalization for $ESL_n (\mathbb{F}_2)$, $n>4$.
\end{proposition}
\begin{proof}
The permutation matrix $P_5 \in SL_5 [\mathbb{Z}]$ can be generated by 2 involutions. The proof is based on the following
 equalities $P_{10}=D_0^{-1}B$, where
\[
    A=D_0=\begin{pmatrix}
        0 & 0 & 0 & 0 & 1 \\
        0 & 0 & 0 & 1 & 0 \\
        0 & 0 & 1 & 0 & 0 \\
        0 & 1 & 0 & 0 & 0 \\
        1 & 0 & 0 & 0 & 0
    \end{pmatrix},
    \quad
    B= -D_0P_{10}=\begin{pmatrix}
        0 & 0 & 0 & 1 & 0  \\
        0 & 0 & 1 & 0 & 0  \\
        0 & 1 & 0 & 0 & 0  \\
        1 & 0 & 0 & 0 & 0  \\
        0 & 0 & 0 & 0 & -1
    \end{pmatrix}.
\]

The generating of transvection by involutions which can be spread on matrices of $ ESL_n [\mathbb{Z}]$ for any $n \in \mathbb{N}$ is below
\[
        i_{11}=\begin{pmatrix}
        -1 & 1 & 0 & 0 & 0  \\
        0 & 1 & 0 & 0 & 0  \\
        0 & 0 & 1 & 0 & 0  \\
        0 & 0 & 0 & 1 & 0  \\
        0 & 0 & 0 & 0 & 1

    \end{pmatrix},
    \quad
    i_{12}=\begin{pmatrix}
        1 &  0 & 0 & 0 & 0  \\
        0 & -1 & 1 & 0 & 0  \\
        0 & 0 & 1 & 0 & 0  \\
        0 & 0 & 0 & 1 & 0  \\
        0 & 0 & 0 & 0 & 1
    \end{pmatrix}.
\]
then $(i_{12}i_{11})^2=t_{13}$ in $\mathbb{F}_2$. Thus, we obtain $t_{13}$ and monomial matrix $P_{10} = B A^{-1}$ of order 10 and $det(P_{10})=-1$.
Therefore, applying Lemma \ref{negativedet}, we prove that the set of involutions $\langle i_{12}, i_{11}, D_0, B \rangle$ generates $ESL_5 (\mathbb{F}_2)$, as well as for $ESL_n (\mathbb{F}_2)$.
\end{proof}

\section*{The case $n = 5$}
\begin{lemma}
The \textbf{minimal generating set} of $ESL_5[\mathbb{Z}]$ is $\langle {M}_{10}, t_{12}\rangle$
\end{lemma}

The permutation realizing by $M_5$ is similar to $(12345)(67)$ with order 10.

\[
t_{12} = \begin{pmatrix}
1 & 1 & 0 & 0 & 0 \\
0 & 1 & 0 & 0 & 0 \\
0 & 0 & 1 & 0 & 0 \\
0 & 0 & 0 & 1 & 0 \\
0 & 0 & 0 & 0 & 1
\end{pmatrix}, \quad
M_5 = \begin{pmatrix}
0 & 0 & 0 & 0 & -1 \\
1 & 0 & 0 & 0 & 0 \\
0 & 1 & 0 & 0 & 0 \\
0 & 0 & 1 & 0 & 0 \\
0 & 0 & 0 & 1 & 0
\end{pmatrix}.
\]
\begin{proof}
The obvious decomposition $t_{ij}= E +e_{ij}$ allows us to analyse only the conjugation of $e_{ij}$ by the matrix $P$, because of $P E P^{-1} =E$.

    Since left acting of permutation matrix $P$ on $e_{ij}$ only permutes rows of matrix $e_{ij}$ then $Pe_{ij}=e_{kj}$, and $e_{ij}P=e_{il}$.
\end{proof}

\section{Minimal involutive generating set.}

\begin{lemma}\label{negativedet}
  \textit{Let $A_1, A_2, ..., A_k \in ESL(n, \mathbb Z)$ be an alphabet $\mathbb{A}$ of matrices, where at least one matrix $A_i$ has a negative determinant. If $SL(n, \mathbb Z) < G = \left\langle A_1, A_2, ..., A_k \right\rangle$, then $G = ESL(n, \mathbb Z)$.}
\end{lemma}
\begin{proof}
Take an arbitrary matrix $B \in ESL(n, \mathbb Z) \setminus SL^{}(n, \mathbb Z)$, and write it in the form $B = A_i^{-1} A_i B$. Therefore, the matrix $A_i B$ belongs to $SL(n, \mathbb Z)$ and consequently, the word over $\mathbb{A}$ expressing the matrix $B$ exists in $G$. Due to the arbitrariness of the matrix $B\in ESL(n, \mathbb Z)$, the entire $ESL(n, \mathbb Z)$ is generated in this way.
\end{proof}

Let ${{t}_{12}}=\left( \begin{matrix}
   1 & 1  \\
   0 & 1  \\
\end{matrix} \right)$, ${{\rho}}=\left( \begin{matrix}
   0 & 1  \\
   1 & 0  \\
\end{matrix} \right)$. 
\begin{proposition}
The minimal generating set of $ES{{L}_{2}}\left( \mathbb{Z} \right)$ is $S=\langle {{t}_{12}},\,\,{{\rho}_{}} \rangle$.
\end{proposition}
\begin{proof}
Applying conjugation by $\rho$ we express
$ \rho t_{12} \rho =  \left(
\begin{matrix}
   1 & 0  \\
   1 & 1  \\
\end{matrix}\right) = t_{21}$, which is the second transvection from ${SL_2(\mathbb{Z})}$.

As is well known the transvections $t_{12}$ and $t_{21}$ generate group $S{{L}_{2}}\left( \mathbb{Z} \right)$ therefore an arbitrary element $C \in SL_2(\mathbb{Z})$ can be expressed.
Having matrix  $\rho$ with $det(\rho)=-1$ we apply Lemma \ref{negativedet} to prove that $S$ is the generic set of $ESL_{2}\left( \mathbb{Z} \right)$.
\end{proof}

An impotent property of involutions is formulated below.

\begin{lemma}
    If involutions $A$, $B$ commute, then $AB$ is an involution too. Otherwise, this is also true.
\end{lemma}

\begin{proof}
Let $i_1i_2=i_2i_1$ then $i_1i_2i^{-1}_1=i_2$, $i_1i_2i^{-1}_1i^{-1}_2=e$ and so $i_1i^{-1}_1i_2i_2=i_1i_1i_2i_2=(i_1i_2)^2=e.$

Vice versa if $(i_1i_2)^2=e$ then $(i_1i_2) (i_1i_2)=i_1i_2 i^{-1}_1i^{-1}_2=[i_1,i_2 ]=e$ so $[i_1,i_2 ]=e$. This completes the proof.
\end{proof}

\begin{example}
For instance, let
${{i}_{12}}=\left( \begin{matrix}
   -1 & 1 & 0  \\
   0 & 1 & 0  \\
   0 & 0 & 1  \\
\end{matrix} \right)$
and
${{i'}_{23}}=\left( \begin{matrix}
   -1 & 1 & 0  \\
   0 & 1 & 1  \\
   0 & 0 &-1  \\
\end{matrix} \right)$,
 then ${{i}_{12}}{i'}_{23}= {{i'}_{23}}{i}_{12} = \left( \begin{matrix}
   1 & 0 & 1  \\
   0 & 1 & 1  \\
   0 & 0 &-1  \\
\end{matrix} \right) = I$ that is involution $I^2=E$.
\end{example}

For convenience, we fix some notation for \textit{diagonal involutive matrices} from $ESL_3(\mathbb{Z})$:
\begin{center}
$I_{12} = \begin{pmatrix}
            -1 & 1 & 0 & 0 & 0 \\
            0  & 1 & 0 & 0 & 0 \\
            0  & 0 & 1 & 0 & 0 \\
            0  & 0 & 0 & 1 & 0 \\
            0  & 0 & 0 & 0 & 1
        \end{pmatrix},
        I_{23} = \begin{pmatrix}
            1 & 0 & 0 & 0 & 0 \\
            0 & -1& 1 & 0 & 0 \\
            0 & 0 & 1 & 0 & 0 \\
            0 & 0 & 0 & 1 & 0 \\
            0 & 0 & 0 & 0 & 1
        \end{pmatrix}, \ldots,
        I_{45} = \begin{pmatrix}
            1 & 0 & 0 & 0 & 0 \\
            0 & 1 & 0 & 0 & 0 \\
            0 & 0 & 1 & 0 & 0 \\
            0 & 0 & 0 & -1& 1 \\
            0 & 0 & 0 & 0 & 1
        \end{pmatrix}.
    $
    \end{center}

\begin{lemma}\label{transvectionoverF_p}
 \textit{For a transvection $T_{ij}(a)$ over the finite ring $\mathbb{Z}_m$, there exists a $k \in \mathbb{N}$ such that $T_{ij}(a)^k$ is an elementary transvection if and only if $\gcd(a, m) = 1$.}
\end{lemma}
\begin{proof}
 If $\gcd(a, m) = 1$ then $a \in \mathbb Z^{\times}$ (integral domain of $\mathbb{Z}_m$), i.e., $a$ is a unit in the ring $\mathbb Z_m$. Then there exists $a^{-1} \in \mathbb Z_m$ such that $a a^{-1} \equiv 1 \mod{m}$. Choosing $k \equiv a^{-1} \mod{m}$ we obtain $T_{ij}(a)^k =T_{ij}(ka) = T_{ij}(1)$.     $\blacksquare$
\end{proof}

\begin{corollary}
\textit{For a transvection $T_{ij}(a)$ $a \in \mathbb{F}^{*}_p$ over the finite field $\mathbb{F}_p$, there exists $k \in \mathbb{N}$ such that $T_{ij}(a)^k$ is an elementary transvection.}
\end{corollary}\label{transvectionoverF_p}

 \begin{theorem}\label{minimal3gen}
A minimal generating set for $ESL_n(\mathbb{F}_2)$, $ESL_n(\mathbb{F}_p)$ with $n > 2$ and $p \in \mathbb{P}$ contains at least 3 generators.

The exceptional case $ESL_2(\mathbb{F}_2)$ is 2-generated group, moreover \\ $ESL_2(\mathbb{F}_2) \simeq D_3 \simeq \langle \rho, t_{12} \rangle$.
\end{theorem}
\begin{proof}
 Due to the well known isomorphism $SL_2(\mathbb{F}_2) \simeq S_3 \simeq D_3 $ and the fact that $1=-1$ in $\mathbb{F}_2$ we have $SL_2(\mathbb{F}_2)= ESL_2(\mathbb{F}_2)\simeq D_3$. As a direct consequence, two involutions generate $ESL_2(\mathbb{F}_2)$. For instance, $ESL_2(\mathbb{F}_2) \simeq \langle \rho, t_{12} \rangle$, note that $t_{12}$ is the involution in $ESL_2(\mathbb{F}_2)$.

To show that $SL_2(\mathbb{F}_p)$ is not a group generated by two involutions, like the dihedral group we show an absence of isomorphism $ESL_2(\mathbb{F}_p)$ with $D_{2p}$.

In order to show that $ESL_2(\mathbb{F}_p)$ is not two involutions generated group as dihedral group, we show an absence of isomorphism $ESL_2(\mathbb{F}_p)$ with $D_{2p}$.

The order of $ESL_2(\mathbb{F}_3)$ is 48.
 $ESL_2(\mathbb{F}_3)$ does not contain an element of order 24, hence it cannot be isomorphic to $D_{48}$ having a cyclic group of order 24. Similarly $SL_2(\mathbb{F}_3)$ has no elements with order 12 so $SL_2(\mathbb{F}_3)$ is not isomorphic to $D_{12}$. Arguing in similar way we justify that $ESL_2(\mathbb{F}_p)$ has not two generating set.

 If $n>3$ then $D_{2p}$ is solvable in contrast with $ESL_2(\mathbb{F}_p)$. That completes the proof.
  \end{proof}

\textbf{Proposition 2.}\label{I_{12},D_{1}, F_L} {\sl The minimal involutive generating set for both $ESL_5\left[ \mathbb{Z}\right]$ and $ESL_5\left(  \mathbb{F}_p \right)$ is composed of 3 involutions. For instance:}

\begin{center}
    $
        I_{12} = \begin{pmatrix}
            -1 & 1 & 0 & 0 & 0  \\
            0  & 1 & 0 & 0 & 0  \\
            0  & 0 & 1 & 0 & 0  \\
            0  & 0 & 0 & 1 & 0  \\
            0  & 0 & 0 & 0 & 1
        \end{pmatrix},
        D_1 = \begin{pmatrix}
            0  & 0  & 0  & 0  & -1 \\
            0  & 0  & 0  & 1 & 0  \\
            0  & 0  & -1 & 0  & 0  \\
            0  & 1 & 0  & 0  & 0  \\
            -1 & 0  & 0  & 0  & 0
        \end{pmatrix},
        F_L = \begin{pmatrix}
            1 & 0 & 0 & 0 & 0 \\
            0 & 0 & 0 & 0 & 1 \\
            0 & 0 & 0 & 1 & 0 \\
            0 & 0 & 1 & 0 & 0 \\
            0 & 1 & 0 & 0 & 0
        \end{pmatrix}.
    $
\end{center}
\begin{proof}
 The key step in the proof is to generate all elementary transvections using the given involutions.
Over this alphabet $I_{12}, D_1, F_L$, there exist words for generating both a transvection and a permutation matrix $P$ of order 5:
\begin{center}
     \begin{align*}
    &   t_{41} = I_{12} D_1 F_L I_{12} D_1 F_L I_{12} D_1 I_{12} D_1 F_L I_{12} D_1 I_{12} D_1 F_L I_{12} D_1 I_{12} D_1 F_L\\
    & I_{12} D_1 F_L D_1 I_{12}, \\
    &    P = \begin{pmatrix}
            0 & 0 & 0 & 1 & 0 \\
            0 & 0 & 1 & 0 & 0 \\
            1 & 0 & 0 & 0 & 0 \\
            0 & 0 & 0 & 0 & 1 \\
            0 & 1 & 0 & 0 & 0
        \end{pmatrix} = I_{12} F_L D_1 F_L I_{12} D_1 I_{12} D_1 F_L D_1 F_L I_{12} D_1 F_L \times \\ & \times D_1 I_{12} F_L
     D_1 F_L I_{12} F_L D_1.
    \end{align*}
\end{center}
Note that the expression for $t_{41}$ consists of 26 letters.
According to \cite{VSEM}, one transvection and a permutation matrix are sufficient to generate $SL(5, \mathbb{Z})$, by virtue of the action by a permutation matrix $P^{-1}t_{41}P=t_{ij}$ we obtain all transvections from $SL(5, \mathbb{Z})$ having a given one.

 Moreover, according to Lemma \ref{negativedet}, after generating $SL(5, \mathbb{Z})$ one additional matrix $I_{12}$ ($\det(I_{12}) = -1$) is sufficient to extend it to $ESL(5, \mathbb{Z})$.


 Furthermore, the important observation $F_L I_{12} F_L = I_{15}$ allows us to generate all involutions,
using the given involution and a conjugation them by $F_L$.

Applying the reduction homomorphism $ESL(5, \mathbb{Z}) \xrightarrow{modp} ESL(5, \mathbb{F}_p)$ gives us the finite group $ESL(5, \mathbb{F}_p)$ as a homomorphic image. Consequently, three images of generators in $ESL(5, \mathbb{F}_p)$ are the same, as initial generators because of reduction of $1$ and $-1$ by $\mod 5$ are the same elements in $\mathbb{F}_5$.

Taking into account Theorem \ref{minimal3gen} these groups do not admit set of two generators.
\end{proof}


\begin{theorem}\label{ID0FU} {\sl The minimal involutive generating set of $ESL\left( 5, \mathbb{Z} \right)$ as well as for $ESL\left( 5, \mathbb{F}_p \right)$ consists of 3 involutions $D_0, F_U, I_{12}$ with the relations $(D_0 F_U)^{4} I_{12} (F_U D_0)^{4}I_{12}^{-1}=E$, $(I_{12} D_0)^4 =E$, $(I_{12} F_U)^4 =E$, $(D_0 F_U )^5=E$, $D_0^2 = F_U^2 = I^2_{12}=E$, where }
    \begin{center}
    $
        I_{12} = \begin{pmatrix}
            -1 & 1 & 0 & 0 & 0  \\
            0  & 1 & 0 & 0 & 0  \\
            0  & 0 & 1 & 0 & 0  \\
            0  & 0 & 0 & 1 & 0  \\
            0  & 0 & 0 & 0 & 1
        \end{pmatrix},
                D_0 = \begin{pmatrix}
            0  & 0  & 0  & 0  & -1 \\
            0  & 0  & 0  & -1 & 0  \\
            0  & 0  & -1 & 0  & 0  \\
            0  & -1 & 0  & 0  & 0  \\
            -1 & 0  & 0  & 0  & 0
        \end{pmatrix},
        F_U = \begin{pmatrix}
            0 & 0 & 0 & 1 & 0 \\
            0 & 0 & 1 & 0 & 0 \\
            0 & 1 & 0 & 0 & 0 \\
            1 & 0 & 0 & 0 & 0 \\
            0 & 0 & 0 & 0 & 1
        \end{pmatrix}.
              $
             \end{center}

\end{theorem}
\begin{proof}   To justify the above relations, we note that these involutions possess the relation of diagonal shift of the involutive cell $D_0 F_U I_{12} F_U D_0=I_{23}$, which implies the relation $(D_0 F_U)^{4} I_{12} (F_U D_0)^{4}=I_{12}$, $I_{12} D_0 I_{12} D_0 I_{12} D_0 I_{12} D_0 =E$, $(I_{12} F_U)^4 =E$, $(D_0 F_U )^5=E$.
The key step in the proof is to generate a transvection using the given involutions, in order to do this we investigate the relation in this generic set. Then $t_{25}$ is expressed by the word of 26 elements:
$
t_{25} = D_0 I_{12} F_U I_{12} D_0 I_{12} F_U I_{12} D_0 I_{12} \\ D_0 F_U I_{12}  D_0 I_{12}  D_0 F_U I_{12} D_0 I_{12} D_0 F_U D_0 F_U I_{12} D_0.
$

Constructing this set consisting of $n$ transvections according to Theorem 2.1 from \cite{Humph} means that we have constructed a generating set. Furthermore, as was studied in \cite{VSEM} a minimal generating set from transvections from $SL(n,\mathbb{Z})$ is of size $n$. This set of generators allows us to express the permutation matrix $P_5$ in the form
\begin{equation*}
    P_5 = F_U \cdot D_1 \cdot F_U \cdot D_1 =
    \begin{pmatrix}
        0 & 0 & 1 & 0 & 0 \\
        0 & 0 & 0 & 1 & 0 \\
        0 & 0 & 0 & 0 & 1 \\
        1 & 0 & 0 & 0 & 0 \\
        0 & 1 & 0 & 0 & 0
    \end{pmatrix}. \end{equation*}

The minimality of this set is based on Corollary \ref{minimal3}.
\end{proof}

\begin{corollary}
If we replace $D_0$ with $-D_0$, that is,
we introduce another matrix $-D_0$ with a negative determinant,
then the resulting set $\{I_{12}, -D_0, F_U\}$ also generates the entire $ESL(5, \mathbb Z)$.
\end{corollary}
Proof.
The set of matrices $F_U, D_0, I_{12}$ was introduced in Theorem \ref{ID0FU}.
Therefore, we can analyze the tuple of matrices $F_U, -D_0, I_{12}$ as reachable using elementary Nielsen transforms.

Firstly, we write the expression for permutation matrix $P$
\begin{equation*}
    P = F_U \cdot (-D_0) \cdot F_U \cdot (-D_0) =
    \begin{pmatrix}
        0 & 0 & 1 & 0 & 0 \\
        0 & 0 & 0 & 1 & 0 \\
        0 & 0 & 0 & 0 & 1 \\
        1 & 0 & 0 & 0 & 0 \\
        0 & 1 & 0 & 0 & 0
    \end{pmatrix}.
\end{equation*}
It turns out that the set of matrices $F_U, -D_0, I_{12}$ can also generate the transvection $t_{41}$, namely by the following word of 24 elements:
\begin{align*}
     & t_{41} =  I_{12}F_U I_{12}(-D_0) I_{12}F_U I_{12}(-D_0) I_{12}(-D_0) F_U I_{12}(-D_0) I_{12}(-D_0) F_U \times \\
     &  \times I_{12}(-D_0) I_{12}(-D_0) F_U (-D_0) F_U I_{12}.
\end{align*}
So we may assert that in this generating set a transvection can be obtained by a word of length less than 24. Thus the matrices $F_U, -D_0, I_{12}$ generate $ESL(5, \mathbb Z)$.

\begin{remark}\label{FL}
The transformation to the previous generating set $\langle F_U, -D_0, I_{12} \rangle$ can be given by the formula: \begin{equation*}
    F_L = D_0 F_U D_0
\end{equation*}
proving that the set $F_L, D_0, I_{12}$ is also the generating set as well as $F_U, D_0, I_{12}$, because we made only equivalent invertible
Nielsen transformations of generators.
Thus, instead $F_U$ we can use the matrix $F_L$.
\end{remark}

\begin{theorem}\label{minimal3gen}
A minimal involutive generating set for $ESL_n(\mathbb{F}_2)$, $ESL_n(\mathbb{F}_p)$ with $n > 2$ contains at least 3 generators.

The exceptional case $ESL_2(\mathbb{F}_2)$ is 2-generated group, moreover \\ $ESL_2(\mathbb{F}_2) \simeq D_3 \simeq \langle \rho, t_{12} \rangle$.
\end{theorem}
\begin{proof}
 Due to the well known isomorphism $SL_2(\mathbb{F}_2) \simeq S_3 \simeq D_3 $ and the fact that $1=-1$ in $\mathbb{F}_2$ we have $SL_2(\mathbb{F}_2)= ESL_2(\mathbb{F}_2)\simeq D_3$. As a direct consequence, two involutions generate $ESL_2(\mathbb{F}_2)$. For instance, $ESL_2(\mathbb{F}_2) \simeq \langle \rho, t_{12} \rangle$, note that $t_{12}$ is the involution in $ESL_2(\mathbb{F}_2)$.

To show that $ESL_2(\mathbb{F}_p)$ is not a group generated by two involutions, like the dihedral group we show an absence of isomorphism $ESL_2(\mathbb{F}_p)$ with $D_{2p}$.

In order to show that $ESL_2(\mathbb{F}_p)$ is not two involutions generated group as dihedral group, we show an absence of isomorphism $ESL_2(\mathbb{F}_p)$ with $D_{2p}$.

The order of $ESL_2(\mathbb{F}_3)$ is 48.
 $ESL_2(\mathbb{F}_3)$ does not contain an element of order 24, hence it cannot be isomorphic to $D_{48}$ having a cyclic group of order 24. Similarly $SL_2(\mathbb{F}_3)$ has no elements with order 12 so $SL_2(\mathbb{F}_3)$ is not isomorphic to $D_{12}$. Arguing in similar way we justify that $ESL_2(\mathbb{F}_p)$ has not two generating set.

 If $n>3$ then $D_{2p}$ is solvable in contrast with $ESL_2(\mathbb{F}_p)$. That completes the proof.
  \end{proof}

\begin{theorem}\label{minimal3gen}
A minimal involutive generating set for $PESL_n(\mathbb{F}_p)$, with $p > 2$ contains at least 3 generators.

The exceptional case $ESL_2(\mathbb{F}_2)$ is 2-generated group, moreover \\ $ESL_2(\mathbb{F}_2) \simeq D_3 \simeq \langle \rho, t_{12} \rangle$.
\end{theorem}
\begin{proof}
As it was already proved in Theorem \ref{minimal3gen} $ESL_n(\mathbb{F}_p)$ possess minimal involutive generating set of 3 involutions. If one assume existence of 2 element generic set for $PESL_n(\mathbb{F}_p)$ then this group have to be isomorphic to $D_{2p}$. This lead us to contradiction with fact of solvability of $D_{2p}$ and non-solvability of $PESL_n(\mathbb{F}_p)$ because we have normal embedding $PSL_n(\mathbb{F}_p)$ in $PESL_n(\mathbb{F}_p)$ as the subgroup of index 2.
The alternating way of proof
 is based on the absence of pairwise collinearity of homomorphic images from the generators of the group $ESL_n(\mathbb{F}_p)$.

  As it was already proved $ESL_n(\mathbb{F}_p)$ possess minimal involutive generating set of 3 involutions two of which commutes, $(2\times2; 2)$-generated. As it was shown the class of these groups is closed with respect to homomorphic images \cite{Mark} therefore $PESL_n(\mathbb{F}_p)$ as quotient group admits 3 generators.

 Due to the well known isomorphism $SL_2(\mathbb{F}_2) \simeq S_3 \simeq D_3 $ and the fact that $1=-1$ in $\mathbb{F}_2$ we have $SL_2(\mathbb{F}_2)= ESL_2(\mathbb{F}_2)\simeq D_3$. As a direct consequence, two involutions generate $ESL_2(\mathbb{F}_2)$. For instance, $ESL_2(\mathbb{F}_2) \simeq \langle \rho, t_{12} \rangle$, note that $t_{12}$ is the involution in $ESL_2(\mathbb{F}_2)$.
\end{proof}

\begin{theorem}
 The minimal involutive generating set for $ESL_5(\mathbb{F}_2)$ is $S= \{I_{12}, -D_0, F_U \}$, where
\end{theorem}
\begin{equation*}
    I_{12} = \begin{pmatrix}
        -1 & 1 & 0 & 0 & 0  \\
        0  & 1 & 0 & 0 & 0  \\
        0  & 0 & 1 & 0 & 0  \\
        0  & 0 & 0 & 1 & 0  \\
        0  & 0 & 0 & 0 & 1
    \end{pmatrix},
        -D_0 = \begin{pmatrix}
        0  & 0  & 0  & 0  & 1 \\
        0  & 0  & 0  & 1 & 0  \\
        0  & 0  & 1 & 0  & 0  \\
        0  & 1 & 0  & 0  & 0  \\
        1 & 0  & 0  & 0  & 0
    \end{pmatrix}.
\end{equation*}
We consider the product $D_0 F_U I_{12}F_UD_0=I_{23}= \begin{pmatrix}
        1  & 0 &  0 & 0 & 0  \\
        0  & -1 & 1 & 0 & 0  \\
        0  & 0 &  1 & 0 & 0  \\
        0  & 0 &  0 & 1 & 0  \\
        0  & 0 &  0 & 0 & 1
    \end{pmatrix}.
 $ Now to express a transvection we multiply this involutions and exponent $I_{12}I_{23}$ to second power over $F_2$:  \\
 $I_{12}I_{23} =\begin{pmatrix}
        -1  & -1 & 1 & 0 & 0  \\
         0  & -1 & 1 & 0 & 0  \\
         0  & 0 &  1 & 0 & 0  \\
         0  & 0 &  0 & 1 & 0  \\
         0  & 0 &  0 & 0 & 1
    \end{pmatrix}$,    \\
   $(I_{12}I_{23})^2 =\begin{pmatrix}
        -1  & -1 & 1 & 0 & 0  \\
         0  & -1 & 1 & 0 & 0  \\
         0  & 0 &  1 & 0 & 0  \\
         0  & 0 &  0 & 1 & 0  \\
         0  & 0 &  0 & 0 & 1
     \end{pmatrix}^2          =\begin{pmatrix}
         1  & 2 & -1 & 0 & 0  \\
         0  & 1 &  0 & 0 & 0  \\
         0  & 0 &  1 & 0 & 0  \\
         0  & 0 &  0 & 1 & 0  \\
         0  & 0 &  0 & 0 & 1
     \end{pmatrix} \cong t_{13}(mod2) =\begin{pmatrix}
         1  & 0 &  1 & 0 & 0  \\
         0  & 1 &  0 & 0 & 0  \\
         0  & 0 &  1 & 0 & 0  \\
         0  & 0 &  0 & 1 & 0  \\
         0  & 0 &  0 & 0 & 1
     \end{pmatrix} $.
\\ By extension dimension of $I_{12}$ this result can be spread on arbitrary $n$. Thus, $S$ is the minimal involutive generating set for $ESL_n(\mathbb{F}_2)$. The permutation matrix $P$ in terms of this generic set takes such a form: $P = F_U \cdot (-D_0) \cdot F_U \cdot (-D_0)$. Thus, we obtain the generating set, its minimality follows from the fact that a group generated be two involutions is a dihedral group. Taking into account that a dihedral group $D_n$ is solvable, whereas $ESL_n(\mathbb{F}_2)$ wherein $n>3$, is not solvable, thence it is not generated by two involutions.


\begin{theorem}
{\sl The minimal involutive generating sets of $ESL_5(\mathbb{Z}), \, k\in \mathbb{N}$ as well of $ESL_5\left(  \mathbb{F}_p \right)$ consist of 3 following involutions $\left\langle D_1, F_U, I_{12}\right\rangle $.}
\end{theorem}
\begin{proof}
 As was studied in \cite{VSEM} a minimal generating set from transvections of $ESL(n,\mathbb{Z})$ has size $n$.
The construction of the transition to the two-element set $\left\langle P, t_{13} \right\rangle $ similar to the minimal generating set from Proposition 1 of generators $P$ and elementary transvection $t_{13}$ is written in the following expressions $t_{13} = I_{12} D_1 F_U I_{12} F_U D_1 I_{12} D_1 F_U I_{12} F_U D_1$,
\begin{center}
         $P_5 = D_1 F_U D_1  F_U = \begin{pmatrix}
        0 & 0 & 0 & 1 & 0 \\
        0 & 0 & 0 & 0 & 1 \\
        1 & 0 & 0 & 0 & 0 \\
        0 & 1 & 0 & 0 & 0 \\
        0 & 0 & 1 & 0 & 0
    \end{pmatrix}$.
\end{center}
The assumption of two generating involutive sets gives us a dihedral group $D_{p}$ which is solvable, unlike the $|ESL _ 5\left( \mathbb{F}_p \right)|$, which contradicts the assumption.

This completes the proof.
\end{proof}

We present new involutions which can be adopted for each dimension and has the form \[
  P_{12} =  \begin{pmatrix} -1 & p-1 & 0 \\ 0 & 1 & 0 \\ 0 & 0 & 1 \end{pmatrix}.
    \]
\begin{theorem}
The minimal involutive generating set for $ESL_n(\mathbb{F}_3)$ consists of 3 involutions $P_{12}, -D_0, F_U$. 
\end{theorem}
\begin{proof} To ensure that a transvection is expressed through the involutions, we constructed over the field $F_3$ the following product $P_{12} (D_0 F_U P_{12}F_U D_0) =$
  \[ =
    \begin{pmatrix} -1 & p-1 & 0 \\ 0 & 1 & 0 \\ 0 & 0 & 1 \end{pmatrix}
    \begin{pmatrix} 1 & 0 & 0 \\ 0 & -1 & p-1 \\ 0 & 0 & 1 \end{pmatrix}
    =
    \begin{pmatrix} -1 & -p+1 & (p-1)^2 \\ 0 & -1 & p-1 \\ 0 & 0 & 1 \end{pmatrix} =P.
\]

Consider the powers of the matrix $P$, namely $P^3$:
\[
P^3 = \begin{pmatrix} -1 & 3 & 2 \\ 0 & -1 & 0 \\ 0 & 0 & -1 \end{pmatrix}
\]
Applying the rules of modular arithmetic over the field $\mathbb{Z}_3$ (where $3 \equiv 0 \pmod 3$ and $2 \equiv -1 \pmod 3$), we obtain the matrix elements:
\[
P^3 \equiv \begin{pmatrix} -1 & 0 & -1 \\ 0 & -1 & 0 \\ 0 & 0 & -1 \end{pmatrix}\pmod 3
\]
This result coincides precisely with the negative elementary matrix (transvection) $t_{13}$:
$P^3 \equiv -t_{13} \pmod 3$. We complement the diagonal with units in the lower corner, thereby increasing the matrix's dimension. This completes the proof.
\end{proof}



\section*{Commuting involutions of type $2\times2, 2$}

Consider the following involutions:
\begin{center}
    $
        I_{23} = \begin{pmatrix}
            1 & 0 & 0 & 0 & 0  \\
            0 & -1 & 1 & 0 & 0  \\
            0 & 0  & 1 & 0 & 0  \\
            0 & 0  & 0 & 1 & 0  \\
            0 & 0  & 0 & 0 & -1
        \end{pmatrix},
        I_{43} = \begin{pmatrix}
            1 & 0 & 0 & 0 & 0 \\
            0 & 1 & 0 & 0 & 0 \\
            0 & 0 & 1 & 0 & 0 \\
            0 & 0 & 1 & -1& 0 \\
            0 & 0 & 0 & 0 & 1
        \end{pmatrix}
            $ and $    I_{23} I_{43} =
    \begin{pmatrix}
        1 & 0 & 0 & 0 & 0 \\
        0 & -1 & 1 & 0 & 0 \\
        0 & 0 & 1 & 0 & 0 \\
        0 & 0 & 1 & -1& 0 \\
        0 & 0 & 0 & 0 & 1
    \end{pmatrix} =
    I_{43} I_{23}.
$
\end{center}

\begin{theorem}\label{Maztriple}
    The set of involutions $\left\langle I_{23, 43}, D_0, F_L \right\rangle $ with two commuting involutions $I_{23, 43}, D_0$ is Mazurov triple \cite{Maz}  generates $ESL_3[\mathbb{Z}]$.
\end{theorem}

Since $[I_{23}, I_{43}]=E$, their product $I_{23} I_{43}= I_{23, 43}$ is also an involution. The Nielsen transformation
$(D_0 F_U) I_{12} ( D_0 F_U)^{-1} = (D_0 F_U) I_{12} (F_U D_0 )=I_{23}$ lead us to generic set $\left\langle I_{23}, D_0, F_L \right\rangle $.
Consider this set of generators $F_L, D_0, I_{23, 43}$, where
\begin{equation*}
    F_L = \begin{pmatrix}
            1 & 0 & 0 & 0 & 0 \\
            0 & 0 & 0 & 0 & 1 \\
            0 & 0 & 0 & 1 & 0 \\
            0 & 0 & 1 & 0 & 0 \\
            0 & 1 & 0 & 0 & 0
        \end{pmatrix}, \quad
    D_0 = \begin{pmatrix}
            0  & 0  & 0  & 0  & -1 \\
            0  & 0  & 0  & -1 & 0  \\
            0  & 0  & -1 & 0  & 0  \\
            0  & -1 & 0  & 0  & 0  \\
            -1 & 0  & 0  & 0  & 0
        \end{pmatrix}.
    \end{equation*}

    Proof. In view of Theorem \ref{ID0FU} $\left\langle I_{12}, D_0, F_L \right\rangle$ generates $ESL_5[\mathbb{Z}]$, then
    Remark \ref{FL} $F_L = D_0 F_U D_0$ that entails Nielsen transformation $F_U = D_0 F_L D_0$ to generic set $\left\langle I_{23, 43}, F_U, D_0 \right\rangle$. Since $I_{32} = F_U  I_{23} F_U$ as well as $I_{43} = D_0  I_{23} D_0 $ also, $I_{23} D_0 I_{23} D_0= I_{23, 43}$, we make an equivalent transformation of Nielsen to the new generating set $\left\langle I_{23, 43}, D_0, F_L \right\rangle $. Taking into account that $\left\langle I_{12}, D_0, F_L \right\rangle $ generates $ESL_5[\mathbb{Z}]$, we deduce that $\left\langle I_{23, 43}, D_0, F_L \right\rangle $ also generates it.

    Moreover, due to the relation we obtain a new triple with\textit{\textbf{ two commuting involutions}}
$$ D^2_0=e, I^2_{23, 43}=e,  F_L,   [D_0, I_{23, 43}]=e,  I_{43} = D_0  I_{23} D_0.$$

Based on $I_{23} D_0 I_{23} D_0= I_{23, 43}$ one can verify that \\
$[D^2_0, I_{23, 43}]= D_0 (I_{23} D_0 I_{23} D_0) D_0(I_{23} D_0 I_{23} D_0)^{-1}  = D_0 e D_0 =e$, $D^2_0 =e$, $I_{23, 43}^2=e$.

    \begin{remark}
    Note that in the generic set $\langle F_L,  I_{23, 54}, D_0 \rangle$ of $ESL_5(\mathbb{F}_p)$, a new type of involution pair commutes, namely $ [F_L,  I_{23, 54}]=e$.
    \end{remark}
\begin{proof}
The proof is similar checking with using the relation $I_{ 54} = F_L I_{23} F_L$.

Another triple $ I_{54}=F_L  I_{23} F_L $ then the commuting pair is
    $[ F_L,  I_{23, 54} ]=e$.
\end{proof}

   The new example is $ \langle I_{12, 54}, D_0, F_L \rangle$ with the commuting property $[I_{12, 54}, D_0 ]=e$.
    If we consider the involution $I_{12}$ and note that $D_0 I_{12}D_0 =I_{54}$ therefore $I_{12, 54} = I_{12} \times I_{54}$. As a result we get the new example of generating set with commuting pair $[I_{12, 54}, D_0 ]=e$.

\begin{corollary}
     The minimal involutive generators set $S$ of $ESL_5(\mathbb{F}_p)$ consisting of $I_{23,43}, F_{L}, D_{0}$ is the Mazurov triple $(2\times2, 2)$ \cite{Maz} as well as set of generators satisfying property $sggi$ \cite{Leem}.
\end{corollary}
\begin{proof}
      The proof is based on directly verification that $[D_0, I_{23, 43}]=e$ and the results established in Theorem \ref{Maztriple}.
\end{proof}

The matrix $F_l$ and $D_0$ admits natural generalization on the case $n=9$.
We will transfer the same notations for the introduced involutive matrices, but constructed in dimension 9.






\begin{remark}
    The set $D_0, I_{12,54}, F_L$ is a generating set for $ESL_n(\mathbb{F}_p)$ with the commutation relation $[I_{12,54}, D_0]=E$.
\end{remark}

\begin{proof} Taking into account that in accordance with Theorem \label{ID0FU} $I_{12}, D_0, F_U$ is the set of generators for $ESL_5(\mathbb{F}_p)$
    the main part of the proof is based on the fact that Nielsen transformations from $I_{12}$ to $ I_{12,54}$ exists due to conjugation by $D_0$, that is, $D_0^{-1} I_{12}  D_0 =I_{54}$.
    The generator $F_U$ can be expressed as $F_U = D_0 F_L D_0$. By multiplying these involutions $I_{12}$, $I_{54}$, we express $I_{12, 54}$.
\end{proof}


{\sl
 Recall that a group is called a
 \textit{string group generated by involutions} $\{ \rho_0, \rho_1, \ldots, \rho_{n-1} \}$ (sggi) if there exists such ordering of the involutions wherein $\rho_i\rho_j = \rho_j\rho_i$ for every $i, j \in \{0, \ldots , n - 1 \}$ such that $| i - j| > 1$ \cite{Leem}.}

 \begin{defn}
   By the \textit{extended group} of \textit{(upper) unitriangular matrices} $EUT_n({\mathbb{F}})$ (over a field $\mathbb{F}$) we mean the unitriangular group $UT_n(\mathbb{F})$ \cite{SusUTn} that admits not only 1 but also -1 on the diagonal.

 We prove that $EUT_3[\mathbb{Z}] \simeq \langle \rho_1, \rho_2, \rho_3, \rho_4 \rangle $ is $sggi$ group \cite{Leem} with the following involutive generating set:
$$
\rho_0 = \begin{pmatrix} -1 & 0 & 0 \\ 0 & 1 & 0 \\ 0 & 0 & 1 \end{pmatrix}, \quad
\rho_1 = \begin{pmatrix} -1 & 1 & 0 \\ 0 & 1 & 0 \\ 0 & 0 & 1 \end{pmatrix}, \quad
\rho_2 = \begin{pmatrix} 1 & 0 & 0 \\ 0 & 1 & 1 \\ 0 & 0 & -1 \end{pmatrix}, \\
\rho_3 = \begin{pmatrix} 1 & 0 & 0 \\ 0 & 1 & 0 \\ 0 & 0 & -1 \end{pmatrix},
$$

{\sl where  $[\rho_1, \rho_3] = e$, $[\rho_0, \rho_3] = e$, $[\rho_0, \rho_2] = e$, furthermore this defines the involution generating set of type $(2 \times 2, 2 \times 2)$.}
 \end{defn}

 {\bf Remark}. {\sl If we set $D_0= \rho_0$, $F_L=\rho_1$, $I_{23, 43}=\rho_2$, then this order of involutions justifies that $ESL_5[\mathbb{Z}]$ and $ESL_5(\mathbb{F}_p)$ are string groups generated by involutions (\textit{$sggi$}) \cite{Leem}.}

  \section{Minimal generating set with a fourth order element.}

Let ${{I}_{11}}=\left( \begin{matrix} -1 & 1  \\
  \, 0 &  1  \\
   \end{matrix} \right)$, $T_4=\left( \begin{matrix} 0 & -1  \\
   1 &  0  \\
   \end{matrix} \right)$
and
$\rho=\left( \begin{matrix} 0 & 1  \\
   1 &  0  \\
   \end{matrix} \right)$.

   \begin{proposition}
   The set $S= \langle T_4, I_{11}, \rho \rangle$ generates $ESL_2(\mathbb{Z})$.
   \end{proposition}
    We show that $S$ is the generic set.
   Squaring $T_4^2$ we get $-E$. To express a transvection we consider the product
 $I_{11}T_4 =t_{12}$ and also $T_4 I_{11} = \left( \begin{matrix} 0 & -1  \\
   -1 &   1  \\
   \end{matrix} \right)$ which be denoted by $t$. Applying $-E$ we get $-E t^{-1}_{}= \left( \begin{matrix} 0 & 1  \\
   1 &   -1  \\
   \end{matrix} \right)$.
     Conjugation of $t_{11}$ lead as to $\rho  t_{12} \rho =t_{21}$.
   As well known the transvections $t_{11}$ and $t_{22}$ generate group $S{{L}_{2}}\left( \mathbb{Z} \right)$ and presents of $\rho$ with $\det(\rho)=-1$ extends this $S{{L}_{2}}\left( \mathbb{Z} \right)$ to $ES{{L}_{2}}\left( \mathbb{Z} \right)$.


Now we return to dimension 3 and consider involutive set of generators for $ES{{L}_{3}}\left( \mathbb{Z} \right)$.
Let ${{P}_{3}}= \begin{pmatrix}
   0 & 1 & 0  \\
   0 & 0 & 1  \\
   1 & 0 & 0  \\
\end{pmatrix} $.
\begin{proposition}
    In terms of generating set $\left<{P}_{3}, t_{12}, D_{123}\right>$ wherein
$t_{12}, t_{32}$ are trans-vections, its
relations are the following: ${P}_{3}t_{12} {P}^{-1}_{3}=t^{}_{31}, \, {P}_{3}t_{31} {P}^{-1}_{3}=t^{-1}_{23} $, ${P}_{3}t_{12} {P}^{-1}_{3}=t^{}_{31}, \, {P}_{3}t_{12} {P}^{-1}_{3}=t^{-1}_{23} $, ${P}^3_{3}=E, [t_{ij}, t_{jk}]=t_{ik}, \textbf{wherein} \, i\neq k,  \,\, [t_{ij}, t_{kl}]=e$ provided $i\neq l$ and $k \neq j$.
\end{proposition}

The proof is an elementary computation and verification.

Note that it is possible to express a transvection using only two non-commutative involutions for a matrices of arbitrary degree $n\in \mathbb{N}, n\geq 3$.

\begin{example}
\begin{center}
${{i}_{12}}=\left( \begin{matrix}
   1 & 1 &  0  \\
   0 & -1 & 0  \\
   0 & 0 & -1  \\
\end{matrix} \right)$,  ${{i}_{23}}=\left( \begin{matrix}
   1 & 0 & 0  \\
   0 & 1 & 1  \\
   0 & 0 & -1  \\
\end{matrix} \right)$,
$({{i}_{12}}{{i}_{23}})^2=\left( \begin{matrix}
   1 & 1 & 1  \\
   0 & -1 & -1  \\
   0 & 0 & 1  \\
\end{matrix} \right)^2= \left( \begin{matrix}
   1 & 0 & 1  \\
   0 & 1 & 0  \\
   0 & 0 & 1  \\
\end{matrix} \right) ={{t}_{31}}^{}$.
 \end{center}
Thus, we generate transvection $t_{31}$ by two involutions $i_{12}$ and $i_{23}$.
\end{example}

From this we obtain the set of generators $S=\left< D_1, {i}_{21}, {i}_{32}, M_6\right>$.
Similarly, for $SL_2[\mathbb{Z}]$ generating set is $S'=\left<  {i}_{12}, {i}_{23}, M_6\right>$.

\begin{proposition}\label{Growth}
There are no generic set of three involutions for $ESL(n,\mathbb{Z})$, one of which commutes with two others.
\end{proposition}
\begin{proof}
A number of matrices with the norm less or equal than $R$ be denoted by $N(R)$.
All words over such generating set $\langle A, B, C \rangle$ that can be constructed over an alphabet from two involutions $A, B$ ($[A, B]$) are represented by one of the next sequences either ${{\text{(}AB\text{)}}^{n}}$ or $B{{(AB)}^{n}}$ or ${{(AB)}^{n}}A$ or $B{{(AB)}^{n}}A$.
It is sufficient to describe the form of the periodic part of these words, viz., ${{(AB)}^{n}}$ then it will be clear that not all the matrices from $ES{{L}_{2}}\left( \mathbb{Z} \right)$ can be represented in this way.

For the proof, a specific form ${{(AB)}^{n}}$ is needed.
An additional part of the proof is based on the fact that the group $ES{{L}_{2}}\left( \mathbb{Z} \right)$ is a two-parameter family. Therefore, it is unlikely to be covered by one-parameter families such as the sequence $AB$.

In view of $AB$ has a determinant equal to $±1$, then it is either similar to $\left( \begin{matrix}
   \pm 1 & 1  \\
   0 & \pm 1  \\
\end{matrix} \right)$  or diagonalizable, and in the last case at least one of the eigenvalues has modulus  $\ge 1$,
however over the ring $\mathbb{Z}$ in a diagonal matrix there only can be elements 1 or -1 in any arrangement.
(If both eigenvalues of $AB$  have modulus 1, and it is diagonalizable but then such a matrix $AB$ will simply have finite order. Such matrix generate cyclic group.)

For a matrix with a Jordan block size of 2 by 2 e.v. are $1$ or $-1$, thence $N(R)=cR$. Since the $n$-th power of such Jordan block is equal to ${AB}^n=\left( \begin{matrix}
1 & n \\
0 & 1 \\
\end{matrix} \right)$ or $\left( \begin{matrix}
-1 & n \\
0 & -1 \\
\end{matrix} \right)$.

Since the $n$-th power of the Jordan block is equal to $(AB)^n$, exactly the first $R$ terms (in a series of power of $(AB)$) are contained in a ball of radius $R$.
This cause that exactly the first $R$ terms are contained in a ball of radius $R$.
This determines the linear growth of the matrix norm $clog _{\alpha} R$.


At the same time, the number of elements $M$ where $M \in ESL_2(\mathbb{Z})$ grows as a function $R \sqrt{R}$, because if we fix first column of $M$, it remains to choose two elements of the second column by $\sqrt{R}$ ways in accordance with the theorem about the prime numbers distribution in order to  Diophantine equation $x \alpha - y \beta =1$ to be solvable over $\mathbb{Z}$.
Note that solvability of $x \alpha - y \beta =1$ is equivalent to $( \alpha, \beta)=1$ and therefore the theorem about prime numbers distribution is applicable to estimating the solutions number.


If the matrix $A \in ESL_2(\mathbb{Z})$ has the diagonal form $A=\left( \begin{matrix}
\alpha & 0 \\
0 & \alpha^{-1} \\
\end{matrix} \right)$, where $ \alpha \in \mathbb{C}$ and $| \alpha | >1$, $|\alpha^{-1} | <1$,
this yields exponential growth of a matrix $(AB)^n$ norm,
then in the series in its powers $A^n$ the following growth function of the number of matrices with norm no greater than $R$ holds: $N(R)=clog_{\alpha}R$, where constant $C$ appears due to equivalence transformation.

Therefore, ${{\left( AB \right)}^{n}}$ yields either linear or exponential growth.

The number of $ESL_2(\mathbb{Z})$ elements with a norm that is no grater than $R$ can be estimated from below as $R^2$ due to the relation on the determinant of these elements decrease $N(R)$ from $R^{k^2}$ to $R^{k^2-k}$.
Consequently, the number $N(R)$ of elements having norm no grater than $R$ generated by $\langle A, B, C \rangle$ which forms only four sequences (${{\text{(}AB\text{)}}^{n}}$, $B{{(AB)}^{n}}$, ${{(AB)}^{n}}A$, $B{{(AB)}^{n}}A$) is less than number elements of $ESL_2(\mathbb{Z})$ with a norm less than or equal to $R$, that's accomplish the proof.
\end{proof}
\begin{corollary}\label{minimal3}
    As a direct corollary we obtain that there is no two involution generating set of $ESL(n,\mathbb{Z})$, $n>2$.
\end{corollary}


\subsection{Commuting involutions in subgroup of $ES{{L}_{2}}(\mathbb{F}_p)$}
{\sl  Recall that a group is called a
 \textit{string group generated by involutions} $\{ \rho_0, \rho_1, \ldots, \rho_{n-1} \}$ (sggi) if there exists such ordering of the involutions wherein $\rho_i\rho_j = \rho_j\rho_i$ for every $i, j \in \{0, \ldots , n - 1 \}$ such that $| i - j| > 1$ \cite{Leem}.}

Assume there is indexed tuble of matrices $\{\rho_i\}$ such, that for all $i, j$ if $|i - j| > 1$ then $\rho_i \rho_j = \rho_j \rho_i$ and such a group calls $sggi$ group according to \cite{Leem}.

 {\bf Definition.}  {\sl By the \textit{extended group} of \textit{(upper) unitriangular matrices} $EUT_n({\mathbb{F}})$ (over a field $\mathbb{F}$) we mean the unitriangular group $UT_n(\mathbb{F})$ [SusUTn] that admits not only 1 but also -1 on the diagonal.}

{\bf Property.}
{\sl $EUT_3[\mathbb{Z}] \simeq \langle \rho_1, \rho_2, \rho_3, \rho_4 \rangle$ with the following generating set of involutions comply with commuting property [Leem].}
 \begin{center} $ {\scriptsize
\rho_0 = \begin{pmatrix} -1 & 0 & 0 \\ 0 & 1 & 0 \\ 0 & 0 & 1 \end{pmatrix}, \quad
\rho_1 = \begin{pmatrix} -1 & 1 & 0 \\ 0 & 1 & 0 \\ 0 & 0 & 1 \end{pmatrix}, \quad
\rho_2 = \begin{pmatrix} 1 & 0 & 0 \\ 0 & 1 & 1 \\ 0 & 0 & -1 \end{pmatrix},
\rho_3 = \begin{pmatrix} 1 & 0 & 0 \\ 0 & 1 & 0 \\ 0 & 0 & -1 \end{pmatrix},
} $
\end{center}
{\sl where  $[\rho_1, \rho_3] = e$, $[\rho_0, \rho_3] = e$, $[\rho_0, \rho_2] = e$, this defines the type of involution generating set  $(2 \times 2, 2 \times 2)$.}

\subsection{Minimal involutive generating set for $ESL_3(\mathbb F_2)$.}
 The minimal of transvections generating $SL(n, F)$ is $n$ that was found in \cite{Humph}.
 Involutive generating sets of linear groups over $F_2$ is subject of interest of many authors \cite{Nuz}.
 Here we find all minimal involutive generating sets, here are three of them:
\begin{center}
    $
        I_{11} = \begin{pmatrix}
            1 & 0 & 0 \\
            0 & 1 & 1 \\
            0 & 0 & 1
        \end{pmatrix}, \quad
        I_{12} = \begin{pmatrix}
            1 & 0 & 0 \\
            1 & 1 & 0 \\
            1 & 0 & 1
        \end{pmatrix}, \quad
        I_{13} = \begin{pmatrix}
            1 & 1 & 0 \\
            0 & 1 & 0 \\
            0 & 1 & 1
        \end{pmatrix}.
    $ \\[6pt]
    $
        I_{21} = \begin{pmatrix}
            1 & 0 & 0 \\
            1 & 1 & 0 \\
            0 & 0 & 1
        \end{pmatrix} \quad
        I_{22} = \begin{pmatrix}
            1 & 0 & 1 \\
            0 & 1 & 1 \\
            0 & 0 & 1
        \end{pmatrix} \quad
        I_{23} = \begin{pmatrix}
            1 & 1 & 1 \\
            0 & 0 & 1 \\
            0 & 1 & 0
        \end{pmatrix}
    $ \\[6pt]
    $
        I_{31} = \begin{pmatrix}
            1 & 0 & 0 \\
            1 & 1 & 1 \\
            0 & 0 & 1
        \end{pmatrix} \quad
        I_{32} = \begin{pmatrix}
            1 & 0 & 1 \\
            0 & 1 & 1 \\
            0 & 0 & 1
        \end{pmatrix} \quad
        I_{33} = \begin{pmatrix}
            1 & 1 & 1 \\
            0 & 0 & 1 \\
            0 & 1 & 0
        \end{pmatrix}
    $
\end{center}
But over $\mathbb{F}_2$ we have $-1\cong 1 mod 2$ then $ESL(3, \mathbb F_2) = SL(3, \mathbb F_2)$.
    Since $ESL(3,F_2)$ has not 2-generated involutive set, thence these sets are minimal as well.

There are 117 involutions in $SL(3, \mathbb F_2)$ and 48672 triples (combinations) of involutions.
Here we find minimal involutive generating set of $ESL(3, \mathbb F_3)$:
\begin{center}
 $        I_{31} = \begin{pmatrix}
            2 & 0 & 0 \\
            0 & 2 & 0 \\
            2 & 2 & 2
        \end{pmatrix}, \quad
        I_{32} = \begin{pmatrix}
            2 & 0 & 2 \\
            0 & 2 & 0 \\
            0 & 0 & 2
        \end{pmatrix}, \quad
        I_{33} = \begin{pmatrix}
            2 & 2 & 2 \\
            0 & 0 & 2 \\
            0 & 2 & 0
        \end{pmatrix}.     $
\end{center}

{\bf Conclusion.}
Size of minimal generated sets of $ESL(5, \mathbb{Z} )$ and $ESL(5, \mathbb{F}_p)$ as involutive as well as not involutive was found  by us in this research.

The Mazurov triples of involutions as the generator systems for $ESL(5, \mathbb{Z} )$ and $ESL(5, \mathbb{F}_p)$ are researched, the minimality of this triples is proved.

\textbf{Sources of Funding} for Research Presented in a Scientific Article or Scientific Article Itself
This work was partially supported by a grant from the \textbf{Simons Foundation} (\textbf{SFI-PD-Ukraine-00017674, Ruslan Skuratovskii}).


\small
\begin{thebibliography}{99}

\bibitem{SkuESL} \emph{Skuratovskii Ruslan, Lysenko S. O.} Extended Special Linear group $ESL_2(F)$ and matrix equations in $SL_2(F)$, $ESL_2(Z)$ and $GL_2(F_p)$. WSEAS TRANSACTIONS on MATHEMATICS DOI: 10.37394/23206.2024.23.68

\bibitem{LatCry}  Ajtai, Miklos  "Generating Hard Instances of Lattice Problems". Proceedings of the Twenty-Eighth Annual ACM Symposium on Theory of Computing. (1996). pp. 99–108. CiteSeerX 10.1.1.40.2489. doi:10.1145/237814.237838. ISBN 978-0-89791-785-8. S2CID 6864824.

 \bibitem{NTRU} FOUQUE, Pierre-Alain et al. Falcon: Fast-Fourier Lattice-based Compact Signatures over NTRU. 2020. Available from the Internet on <https://falcon-sign.info/>, accessed in November 8th, 2020.

\bibitem{BookLatCry} Güneysu, Tim; Lyubashevsky, Vadim; Poppelmann, Thomas (2012). Practical Lattice-Based Cryptography: A Signature Scheme for Embedded Systems (PDF). Cryptographic Hardware and Embedded Systems --- CHES 2012. Lecture Notes in Computer Science. Vol. 7428. IACR. pp. 530-547. doi:10.1007/978-3-642-33027-8 31.

\bibitem{Amit} \emph{Amit Kulshrestha and Anupam Singh.} Computing $n$-th roots in $SL_2(Z)$ and Fibonacci polynomials.
{\it Proc. Indian Acad. Sci.} (Math. Sci.) (2020) 130:31 https://doi.org/10.1007/s12044-020-0559-8.

\bibitem{Levch} \emph{Levchuk, D. V.} On generation of the group $PSL_n(Z+iZ)$ by three involutions,
two of which commute / D. V. Levchuk, Ya. N. Nuzhin // Journal SFU. Serie Math-Ph.
 2008. V. 1, Num. 2. pp. 133-139.

\bibitem{Coxet}  \emph{John Frank Adams.} Lectures on Lie Groups Second Edition by University of Chicago Press,  Revised ed. January 15, 1983. 192 pages.


\bibitem{TP} \emph{T. E. Panov and Ya. A. Veryovkin}. Polyhedral products and commutator subgroups of right-angled Artin and Coxeter groups. Sbornik: Mathematics 207:11 \, 1582-1600.

 \bibitem{Maz} \emph{Mazurov, V. D.} The Kourovka notebook: Unsolved Problems in Group Theory
/ Eds. V. D. Mazurov, E. I. Khukhro // Sobolev Institute of Mathematics,
Novosibirsk, 2022, Num. 20.


\bibitem{Klyach} \emph{Klyachko Anton A., Baranov D. V.} Economical adjunction of square roots to groups. Sib. math. journal, Volume 53 (2012), Number 2, pp. 250-257.

 \bibitem{Suds} \emph{H. A. Janabi, L. Hethelyi and E. Horvoth} (2020)
     {\em Journal of Group Theory.}
     TI subgroups and depth 3-subgroups in simple Suzuki groups.
     https://doi.org/10.1515/jgth-2020-0044


\bibitem{Zu} \emph{ N.~D.~Zyulyarkina}, On the commutation graph of cyclic TI-subgroups in linear groups, // Proc. Steklov Inst. Math. (Suppl.), 279, suppl. 1 (2012), 175-181.

\bibitem{Nuz}
\emph{Ya. N. Nuzhin}, Generating triples of involutions of groups of Lie type over a finite field of odd characteristic. II, Algebra and Logic, 36:4(1997), 422-440.

\bibitem{LinearShpringer}
\emph{Jurg Liesen, Volker Mehrmann}. {\em Linear Algebra. Springer Undergraduate Mathematics Series. Springer International Publishing Switzerland 2015 (2015).} DOI https://doi.org/10.1007/978-3-319

\bibitem{Mersl}
\emph{ Yu.~I.~Merzlyakov}, Automorphisms of two-dimensional congruence groups, Algebra and Logic, 10.1007/BF02218574, 12, 4, (262-267), (1973).

\bibitem{Humph}
\emph{Stephen~P.~Humphries}. Generation of Special Linear Groups by Transvections. Journal
of Algebra 99 (1986), p. 480 - 495.



\bibitem{VSEM} \emph{M.~A.~Vsemirnov}.
ON (2,3)-GENERATION OF MATRIX GROUPS
OVER THE RING OF INTEGERS. {\it St. Petersburg Math. J.}
 (2008),  Vol. 19 No. 6, 883--910.

\bibitem{Leem} \emph{ Dimitri~Leemans}. String C-group representations of almost simple groups: A
survey. {\it Contemporary Mathematics }
Volume 764, 2021 https://doi.org/10.1090/conm/764/15335.

\bibitem{SusUTn}
\emph{A.~S.~Oliinyk, V.~I.~Sushchanskii}, Free group of infinite unitriangular matrices,
Mat. Zametki, 2000, Volume 67, Issue 3, 382 - 386.

\bibitem{Mark}
Irina A. Markovskaya, Yakov N. Nuzhin.
 On Generation of the Groups $GL_n(Z)$ and $PGL_n (Z)$ by Three
Involutions, Two of which Commute. / Journal of Siberian Federal University. Mathematics \& Physics 2023, Num. 16(4), 413 - 419.


\bibitem{Aschbacher} \emph{ M. Aschbacher}.
Finite Group Theory. August 2, 2010 Cambridge Studies in Advanced Mathematics 10)   2nd Edition, 318 pages.
\end{thebibliography}
\end{document}